%% file: MAIN.tex
\documentclass[11pt,a4paper]{amsart}

\input{preamble.tex}

\begin{document}

  \input{./TeX/title.tex}

  \begin{abstract}
    \input{./TeX/abstract.tex}
  \end{abstract}

  \maketitle

  \section*{Introduction}
    \label{section:Introduction}
    \input{./TeX/introduction.tex}

  \section{Grothendieck Descent and Codescent}
    \label{section:Grothendieck-Descent-Codescent}
    \input{./TeX/descentcodescent.tex}

  \section{Fully faithful lax epimorphisms}
    \label{section:Fully-faithful-lax-epimorphisms}
    \input{./TeX/fflep.tex}

  \section{Discrete and Split Fibrations}
    \label{section:Discrete-and-Split-Fibrations}
    \input{./TeX/characterization1.tex}

    \label{section:Beck-Chevalley}
    \input{./TeX/Beck.tex}

  \section{Future work}
    \label{section:Epilogue}
    \input{./TeX/final.tex}

  \section*{Acknowledgments}
    \label{section:Acknowledgments}
    \input{./TeX/acknowledgment.tex}


  \bibliographystyle{plain-abb}
  \bibliography{bibliography.bib}

\end{document}

%% file: preamble.tex
\usepackage[T1]{fontenc}

\usepackage{mathtools}
\usepackage{amssymb}
\usepackage{amsthm}
\usepackage{mathrsfs}
\usepackage{tikz-cd}
\usepackage{enumitem}
\usepackage{indentfirst}
\usepackage{stackrel}
\usepackage{csquotes}

\usepackage[geometry]{ifsym}
\usepackage{geometry}

\usepackage[portuguese,english]{babel}

\usetikzlibrary{babel}

\renewcommand{\epsilon}{\varepsilon}

\newcommand{\cat}[1]{\ensuremath{\mathcal{#1}}}

\newcommand{\catname}[1]{\mathsf{#1}}
\newcommand{\catt}[1]{\mathsf{#1}}
\newcommand{\Set}{\catname{Set}}

\newcommand{\Cat}{\catname{Cat}}
\newcommand{\CAT}{\catname{CAT}}

\newcommand{\VCat}{\cat V\dash\Cat}
\newcommand{\VCAT}{\cat V\dash\CAT}

\newcommand{\op}{\mathrm{op}}

\newcommand{\pF}[1]{\mathfrak{#1}}
\newcommand{\pFun}[2]{\pF{#1}\left( {#2}\right) }
\newcommand{\twcat}{\mathcal{A} }

\newcommand{\DESC}[2]{\mathsf{Desc}_{#1}\left( #2 \right)}
\newcommand{\CODESC}[1]{\mathsf{CoDesc}\left( #1 \right)}

\newcommand{\CKomp}[1]{\mathsf{K}^{#1}}

\newcommand{\Komp}[2]{\mathcal{K}_{#1}^{#2}}
\newcommand{\ForD}[2]{\mathsf{d}_{#1}^{#2}}

\newcommand{\EQ}[1]{\mathsf{Eq}\left({#1}\right)}
\newcommand{\EQP}{\EQ{p}}
\newcommand{\iso}{\cong}

\newcommand{\adj}{\dashv}
\newcommand{\dash}{\text{-}}

\DeclareMathOperator{\LKan}{\mathsf{LKan}}

\DeclareMathOperator{\Cauchy}{\mathfrak C}

\newtheorem{lemma}{Lemma}[section]
\newtheorem{proposition}[lemma]{Proposition}

\newtheorem{theorem}[lemma]{Theorem}

\theoremstyle{definition}

\newtheorem{remark}[lemma]{Remark}

%% file: TeX/title.tex
\title[Effective Descent for Split
Fibrations]{Cauchy Completeness, Lax Epimorphisms and Effective Descent for
Split Fibrations}

\author[F. Lucatelli Nunes]{Fernando Lucatelli Nunes$^1$}
\author[R. Prezado]{Rui Prezado$^2$}
\author[L. Sousa]{Lurdes Sousa$^3$}

\address[1]{Utrecht University, The Netherlands}
\address[1,2,3]{University of Coimbra, CMUC, Department of Mathematics,
	Portugal}
\address[3]{Polytechnic Institute of Viseu, ESTGV, Portugal}

\address{}

\email[1]{f.lucatellinunes@uu.nl}
\email[2]{rui.prezado@student.uc.pt}
\email[3]{sousa@estv.ipv.pt}

\thanks{The research was supported through the programme ``Oberwolfach Leibniz
Fellows'' by the Mathematisches Forschungsinstitut Oberwolfach in 2022, and
partially supported by the Centre for Mathematics of the University of Coimbra
- UIDB/00324/2020, funded by the Portuguese Government through FCT/MCTES.  The
second author was supported by the grant PD/BD/150461/2019 funded by Fundação
para a Ciência e Tecnologia (FCT)}

\keywords{Cauchy completions, lax epimorphisms, effective descent morphisms, fully faithful morphisms, enriched categories, split fibrations}

\subjclass{18A22, 18F20, 18D20, 18A20, 18N10}

%% file: TeX/abstract.tex
For any suitable monoidal category $\cat V $, we find that $\cat V
$-fully faithful lax epimorphisms in \( \cat V \dash \Cat \) are precisely
those \( \cat V \)-functors \( F \colon \cat A \to \cat B \) whose induced
$\cat V $-functors  $\Cauchy F \colon \Cauchy \cat A \to \Cauchy \cat B $
between the Cauchy completions are equivalences. For the case $\cat V = \Set
$, this is equivalent to requiring that the induced functor $F^* \colon
\CAT(\cat A,\Cat) \to \CAT(\cat B, \Cat)$ is an equivalence.  

By reducing the study of effective descent functors with respect to the
indexed category of split (op)fibrations  \(\cat F\) to the study of the
codescent factorization, we find that the observations above on fully faithful
lax epimorphisms provide us with a characterization of  (effective) \(\cat
F\)-descent morphisms in the category of small categories \(\Cat \); namely,
we find that they are precisely the (effective) descent morphisms with respect
to the indexed categories of \textit{discrete} opfibrations --- previously
studied by Sobral. We include some comments on the Beck-Chevalley condition
and future work. 

%% file: TeX/introduction.tex
Let $\catt C $ be a category with pullbacks and $p: e\to b $ a morphism in
$\catt C $.  The kernel pair of $p $ induces the internal groupoid $\EQP$,
whose underlying truncated simplicial object in $\catt C $ is given by the
diagram \eqref{eq:kernel-pair-internal-groupoid} below. In this setting, each
indexed category $\pF F : \catt{C} ^\op \to \CAT $ has the associated category
$\DESC{\pF F}{p}$ of internal $\pF F $-actions of the internal groupoid
$\EQP$, also called the \textit{category of $\pF F $-descent data for $p$}.
This category comes with a factorization given by the diagram
\eqref{eq:descent-factorization-introduction} below, where $\ForD{\pF F}{p}$
is the forgetful functor that discards descent data.  \\
\noindent\begin{minipage}{.5\linewidth}
  \begin{equation}\label{eq:kernel-pair-internal-groupoid}
    \begin{tikzcd}
      e \times_b e \times_b e \ar[r,shift left=2]
                              \ar[r,shift right=2]
                              \ar[r]
      & e \times_b e \ar[r,shift left=2,"\pi_e"]
                     \ar[r,shift right=2,"\pi^e",swap]
      & e \ar[l]
    \end{tikzcd}
  \end{equation} 	
\end{minipage}%
\begin{minipage}{.5\linewidth}
  \begin{equation}\label{eq:descent-factorization-introduction}
    \begin{tikzcd}[column sep=small]
      \pFun{F}{b} \ar[rr,"\pFun{F}{p}"] 
                  \ar[rd,"\Komp{\pF F}{p}",swap]
      && \pFun{F}{e} \\
      & \DESC{\pF F}{p} \ar[ru,"\ForD{\pF F}{p}",swap]
    \end{tikzcd}
  \end{equation} 
\end{minipage}\\
In the context of Janelidze-Galois Theory (\textit{viz.} \cite{MR1822890}) and
Grothendieck Descent Theory (\textit{viz.}  \cite{MR1466540, MR3806333}), one
is interested in characterizing the morphisms $p$ in $\catt C $ such that the
comparison $\Komp{\pF F}{p} $ is an equivalence (or just fully faithful); that is to
say, in characterizing the \textit{effective $\pF{F} $-descent (respectively,
$\pF{F} $-descent) morphisms of $\catt C$}. We refer the reader to
\cite{MR1285884, MR2056587} for comprehensive introductions.

By the Bénabou-Roubaud Theorem (see \cite{MR255631} or
\cite[Theorem~8.5]{MR3806333} for a generalization), whenever $\pF{F}$ comes
from a bifibration satisfying the Beck-Chevalley condition (see, for instance,
\cite[Section~4]{MR4266479}), $p$ is of effective $\pF{F} $-descent if and
only if $\pFun{F}{p}$ is monadic. This provides us with a way of studying
(effective) $\pF{F}$-descent morphisms via the Beck's monadicity theorem.

If $\pF{F}$ does not satisfy the Beck-Chevalley condition, the equivalence
will not necessarily hold. A particular prominent example of this setting is the
indexed category of discrete (op)fibrations \( \cat F_D = \CAT(-,\Set) \),
thoroughly studied by Sobral in \cite{MR2107401} (see also
\cite[Remark~4.8]{MR4266479}). 

The main point of this work is to extend Sobral's techniques and viewpoint 
on discrete (op)fibrations~\cite{MR2107401} to other settings. In
this paper, we study the case of the indexed category \( \cat F \colon \Cat
^\op \to \CAT, e\mapsto \CAT(\catt e,\Cat) \); by Grothendieck's construction,
is essentially the indexed category of split opfibrations, which is another
glaring example of an indexed category that does not satisfy the
Beck-Chevalley condition. 

Recall that a lax epimorphism \( p \colon e \to b \) in \( \VCat \) is a \(
\cat V \)-functor such that 
\begin{equation}
  \label{eq:lax.epi.def} 
  \VCat(p, x) \colon \VCat (b, x) \to \VCat (e, x)  
\end{equation} 
is fully faithful for any small \( \cat V \)-category \( x \).  Underlying
Sobral's study of  effective $\cat F _D $-descent morphisms $p$, there are two
fundamental steps.  The first step is to construct a factorization of $p$ such
that its image is  (isomorphic to)
\eqref{eq:descent-factorization-introduction}.  The second step of
\cite{MR2107401} relies on the characterization of lax epimorphisms in
\(\Cat\), that is, the functors \(p\) such that \(\Cat(p, c)\) is fully
faithful for any \(c\).

We revisit these two fundamental steps of \cite{MR2107401}, giving a
systematic view over them, in Sections
\ref{section:Grothendieck-Descent-Codescent} and
\ref{section:Fully-faithful-lax-epimorphisms}. Since it is suitable for our
future work, we do that in the \( \cat V \)-enriched setting. The reader,
however, can opt to always consider the case \( \cat V=\Set \). 

In Section \ref{section:Grothendieck-Descent-Codescent}, we show how we can
reduce the problem of studying effective $\cat F $-descent morphisms to the
study of the codescent factorization induced by
\eqref{eq:kernel-pair-internal-groupoid}: namely, the factorization given by
the universal property of the weighted colimit called \textit{codescent
category} (see, for instance, \cite{MR1935980} or \cite[Lemma 3.3]{MR3833110}).
This observation leads to a straightforward formal result that characterizes
effective $\pF{F}$-descent morphisms whenever the domain of $\pF{F}$ is a
$2$-category with codescent objects and $\pF{F}$ preserves suitable
two-dimensional limits: namely, descent objects.

In Section \ref{section:Fully-faithful-lax-epimorphisms}, we thoroughly 
study the characterization of fully faithful lax epimorphisms in \(\VCat\).
Inspired by its relation with the (co)presheaf categories,  we show its
relation with the Cauchy completion pseudofunctor, which we denote by \(
\Cauchy \). We prove the equivalence of the following statements (Theorem
\ref{thm:sect2}) for monoidal categories \( \cat V \):
\begin{itemize}[noitemsep]
  \renewcommand{\labelitemi}{--}
  \item
    \( p \) is a fully faithful lax epimorphism.
  \item
    \( \Cauchy p \) is an equivalence.
  \item
    \( \VCAT(p, \cat V) \) is an equivalence.
\end{itemize}

In Section \ref{section:Discrete-and-Split-Fibrations}, the characterization
of fully faithful lax epimorphisms together with the formal result of Section
\ref{section:Grothendieck-Descent-Codescent} provide us with a proof that a
morphism is of \textit{(effective) \(\cat F\)-descent if and only if it is of
(effective) \( \cat F_D \)-descent.} This means that Sobral's characterization
can be plainly extended to the case of the indexed category of split fibrations. 

We finish Section \ref{section:Discrete-and-Split-Fibrations} by discussing a
straightforward example related to the well-known fact that our indexed
category indeed does not satisfy the Beck-Chevalley condition.

We end the paper with Section \ref{section:Epilogue}. It gives a brief account
of some problems following this line of work: for example, we state open
problems in the enriched setting and the  $( \cat T,\cat V)$-categorical
setting (\textit{viz.}  \cite{MR1957813, MR2491799}).

%% file: TeX/descentcodescent.tex
In this section, we consider an arbitrary $2$-category \(\twcat\) with lax
codescent objects, but the reader may safely assume \(\twcat = \Cat\), which is
the scope of our main results in
Sect.~\ref{section:Discrete-and-Split-Fibrations} (see, for instance,
\cite{MR1935980} or \cite[p.~42]{MR3833110} for definitions of the
\textit{two-dimensional colimit} known as \textit{lax codescent category}).

Given a $2$-functor \( \pF{F}: \twcat^\op \to \CAT \) and a morphism \( p
\colon e\to b\) in \( \twcat \), we consider the image of the diagram
\eqref{eq:kernel-pair-internal-groupoid} by \(\pF{F}\): namely, the diagram
\eqref{eq:the-image-of-the-internal-groupoid-by-f} below. The universal
property of the \textit{lax descent object} 
\eqref{eq:the-image-of-the-internal-groupoid-by-f} in $\CAT $  induces the
\textit{descent factorization} \eqref{eq:descent-factorization-introduction}
of $\pFun{F}{p} $  (see, for instance, \cite[Lemma~3.6]{MR4266479} for this
description via the \textit{two-dimensional limit lax descent object}). 
\textit{We say that such a morphism \(p\) is of effective \( \pF F \)-descent
(\( \pF F \)-descent) whenever \( \Komp{}{\EQP} \) is an equivalence (resp.
fully faithful).}

We get the factorization of \(p\) in $\twcat$, depicted in diagram
\eqref{eq:descent-factorization-introduction-codescent} below, by the
\textit{lax codescent object} $\CODESC{\EQP}$ of the diagram
\eqref{eq:kernel-pair-internal-groupoid} in $\twcat$. 

\noindent
\begin{minipage}{.4\linewidth}
  \begin{equation}
    \label{eq:descent-factorization-introduction-codescent}
    \begin{tikzcd}[column sep=tiny]
      e \ar[rr,"p"] \ar[rd,"\ForD{}{\EQP}",swap] && b \\
      & \CODESC{\EQP} \ar[ru,"\Komp{}{\EQP}",swap]
    \end{tikzcd}
  \end{equation} 
\end{minipage}
\begin{minipage}{.6\linewidth}
  \begin{equation}
    \label{eq:the-image-of-the-internal-groupoid-by-f}
    \begin{tikzcd}[column sep=small]
      \pFun{F}{e} \ar[rr,shift left=2,"\pFun{F}{\pi_e}"]
                  \ar[rr,shift right=2,"\pFun{F}{\pi^e}",swap]
      && \pFun{F}{e \times_b e} \ar[r,shift right=2]
                               \ar[r]
                               \ar[r,shift left=2]
                               \ar[ll]
      & \pFun{F}{e \times_b e \times_b e} 
    \end{tikzcd}
  \end{equation} 	
\end{minipage}

If $\pF{F}: \twcat^\op \to \CAT$ preserves lax descent objects, then the
image of \eqref{eq:descent-factorization-introduction-codescent} by $\pF{F}$
is isomorphic to \eqref{eq:descent-factorization-introduction}. In
particular:
\begin{lemma}\label{lemma:fundamental-basic}
  Let $\pF{F}: \twcat ^\op \to \CAT$ be a $2$-functor that preserves
  two-dimensional limits. A morphism $p:e\to b $ is of effective \( \pF{F}
  \)-descent (\( \pF{F} \)-descent) if, and only if, $\pFun{F}{\CKomp{\EQP}} $
  is an equivalence (fully faithful). 
\end{lemma}



%% file: TeX/fflep.tex
\textit{Throughout this section, let \( \cat V \) be a symmetric monoidal
closed, complete and cocomplete category.} We consider the $2$-category \(
\VCat \) of small \( \cat V \)-categories.

A morphism \( p \colon e \to b \) in a $2$-category $\cat A$ is a \textit{lax
epimorphism } if $\cat A (p, c) $ is fully faithful in $\CAT $ for any $c\in
\cat A $. This is the dual of the notion of fully faithfulness.  In
particular, a \( \cat V \)-functor \( p \colon e \to b \) is a \textit{lax
epimorphism} whenever \eqref{eq:lax.epi.def}  is fully faithful for every
small \( \cat V \)-category \( x \).

A morphism $p \colon e \to b $  of $\VCat $ is \textit{$\cat V$-fully
faithful} if, for any $x, y\in e $, the morphism $p : e(x, y) \to b( px, py) $
in $\cat V $ is invertible. It is easy to see that, if $q$ has an adjoint in
$\VCat$, $q$ is a $\cat V $-fully faithful morphism if, and only if, $q$ is
fully faithful in the $2$-category $\VCat $ (see
\cite[Lemma~5.2]{lucatelli_sousa}).

The main point of this section is to study characterizations of the morphisms
that are simultaneously $\cat V$-fully faithful and lax epimorphic in \( \VCat
\). In particular, we give a characterization in terms of the Cauchy
completions of the \( \cat V \)-categories. We recall basic aspects about
those below.

\subsection{Cauchy completion}

An object \(a\) of a (possibly large) \( \cat V \)-category \( \catt C \) is
\textit{tiny} if the \( \cat V \)-functor \(\catt C(a,-) \) preserves
colimits (see tiny in \cite{BD86}, or \textit{small-projective} in
\cite{Kel05}). For a small \( \cat V \)-category \(e\), we denote by $\Cauchy
e $ the full \( \cat V \)-subcategory of \( \VCAT[e, \cat V] \) consisting of the
tiny objects of \( \VCAT[e, \cat V] \).

\textit{ Henceforth, we assume that $\cat V $ is such that $\Cauchy e $,
called the \textit{Cauchy completion of $e$}, is essentially small for any  \(
e\in \VCat \)}. This is true for many base categories \( \cat V \). We are
mainly interested in the cases \( \cat V = \Set \) and \( \cat V = \Cat \);
other examples include the extended real line, or even more generally any
small quantale.

Recall that \textit{tiny objects are preserved by equivalences}. More
generally, we have:

\begin{lemma}
  Let $\left( F\dashv G\right) \colon \catt C \to \catt D $  be a $\cat V
  $-adjunction between (possibly large) $\cat V $-categories. If  $G$ is
  colimit-preserving, then $F$ preserves tiny objects. 
\end{lemma} 	
\begin{proof}
  Indeed, if \( a \) is a tiny object,  then \( \catt D(Fa,-) \iso \catt
  C(a,G(-)) \) is colimit-preserving since it is a composite of
  colimit-preserving functors.
\end{proof}

For a functor \( p \colon e \to b \) in \( \VCat \), we denote by \( \Cauchy
p \colon \Cauchy e \to \Cauchy b \) the $\cat V$-functor induced by the
restriction of the left Kan extension \( \LKan_p \colon \VCAT[e ,\cat V] \to
\VCAT[b, \cat V] \) to the tiny objects. It is clear that $\Cauchy$ naturally
extends to a pseudofunctor $\VCat \to \VCat$.

\subsection{The characterization}

We start by establishing characterizations  of $\cat V$-fully faithful
morphisms and lax epimorphisms of $\VCat$ in terms of the induced morphism
between the Cauchy completions, and in terms of the induced functor between
the categories of $\cat V$-presheaves.

For a small $ c $, we denote by $\eta _c $ the full inclusion (induced by the
Yoneda embedding) of the $\cat V$-category $c$ into its Cauchy completion.
This defines a natural transformation $\mathrm{Id} \to \Cauchy $. Moreover,
recall that, by the universal property of the  Cauchy completion, the vertical
arrows of diagram \eqref{eq:cauchy.square} below are equivalences. Therefore:
 
\begin{lemma}\label{lem:cauchy-presheaves}
  If $p$ is a morphism of $\VCat$, then the induced functor between the
  categories of $\cat V$-presheaves ${\VCAT(p,\cat V)}$ is fully faithful
  (resp. lax epimorphic) if and only if $\VCAT( \Cauchy p,\cat V)$  is fully
  faithful (resp. lax epimorphic) as well.
\end{lemma}

\begin{equation}
  \label{eq:cauchy.square}
  \begin{tikzcd}[column sep=large]
    \VCAT(\Cauchy b, \cat V) 
      \ar[rr,"{\VCAT(\Cauchy p,\cat V)}"] 
      \ar[d, swap,"{\VCAT(\eta_b,\cat V)}"]
    && \VCAT(\Cauchy e, \cat V) 
      \ar[d,"{\VCAT(\eta_e,\cat V)}"] \\
    \VCAT(b, \cat V) \ar[rr,swap, "{\VCAT(p,\cat V)}"]
    && \VCAT(e, \cat V)
  \end{tikzcd}
\end{equation}

By the above and \cite[Theorem~5.6]{lucatelli_sousa}, we get, then, a full
characterization of lax epimorphisms in terms of $\Cauchy$. More precisely:

\begin{proposition}
  \label{lemma:le}
  The following conditions are equivalent for a \( \cat V \)-functor \( p
  \colon e \to b \) between small \( \cat V \)-categories:
  \begin{enumerate}[label=\emph{\roman*.},noitemsep]
    \item \label{eq:lemma:le1} 
      \(p\) is a lax epimorphism.
    \item\label{eq:lemma:le2} 
      \( \Cauchy p \) is a lax epimorphism.
    \item\label{eq:lemma:le3} 
      \( \VCAT(p, \cat V) \) is fully faithful.
  \end{enumerate}
\end{proposition}
\begin{proof}
  The equivalence of \ref{eq:lemma:le1}  and \ref{eq:lemma:le3} was already
  established by (b) \( \leftrightarrow \) (c) of
  \cite[Theorem~5.6]{lucatelli_sousa}.  The equivalence of \ref{eq:lemma:le1}
  and \ref{eq:lemma:le2} follows from Lemma \ref{lem:cauchy-presheaves} and
  the equivalence \ref{eq:lemma:le1} \( \leftrightarrow \) \ref{eq:lemma:le3}. 
\end{proof}

Although it does not follow from (plain) duality, the counterpart of
Proposition \ref{lemma:le} holds for fully faithful morphisms. 

We start by recalling that  \textit{the unit of an adjunction \( f \adj g \)
in a $2$-category $\cat A $  is invertible iff \( f \) is fully faithful iff
\( g \) is a lax epimorphism.} Moreover, by coduality, the counit is an
isomorphism iff \(f\) is a lax epimorphism iff \(g\) is fully faithful. In
particular, we have that, \textit{assuming that a morphism $p$ has an adjoint,
it is an equivalence if, and only if, it is a fully faithful lax epimorphism}.
All these statements hold for $\cat A = \VCat $ when replacing fully
faithfulness by $\cat V$-fully faithfulness since, for adjoints,
both are equivalent (see \cite[Examples~2.4, (3)]{lucatelli_sousa} and
\cite[Lemma~5.2]{lucatelli_sousa}).

\begin{proposition}
  \label{lemma:ff}
  The following conditions are equivalent for a \( \cat V \)-functor \( p
  \colon e \to b \) between small \( \cat V \)-categories:
  \begin{enumerate}[label=\emph{\roman*.},noitemsep]
    \item\label{lemma:ffi}
      \(p\) is $\cat V$-fully faithful.
    \item\label{lemma:ffii}
      \( \Cauchy p \) is $\cat V$-fully faithful.
    \item\label{lemma:ffiii}
      \( \VCAT(p,\cat V) \) is a lax epimorphism.
  \end{enumerate}
\end{proposition}
\begin{proof}
  As in the case of Proposition \ref{lemma:le}, by Lemma
  \ref{lem:cauchy-presheaves}  it is enough to prove that \ref{lemma:ffi} and
  \ref{lemma:ffiii} are equivalent.
 
  Recall that we have an adjunction $\LKan_p \adj \VCAT(p,\cat V ) $. Hence,
  \( \VCAT(p,\cat V) \) is a lax epimorphism if and only if $\LKan_p $ is
  fully faithful. We complete the proof by observing that, as a consequence of
  the Yoneda Lemma, \( p \) is $\cat V$-fully faithful if and only if \(
  \LKan_p \) is fully faithful (see, for instance, \cite[Proposition
  4.23]{Kel05}). 
\end{proof}

Finally, combining the characterizations of $\cat V$-fully faithful morphisms
and lax epimorphisms, we get:

\begin{theorem}
  \label{thm:sect2}
  The following conditions are equivalent for a \( \cat V \)-functor \( p
  \colon e \to b \) between small \( \cat V \)-categories:
  \begin{enumerate}[label=\emph{\roman*.},noitemsep]
    \item\label{thm:sect2i}
      \(p\) is a $\cat V$-fully faithful lax epimorphism.
    \item\label{thm:sect2ii}
      \( \Cauchy p \) is a $\cat V$-fully faithful lax epimorphism.
    \item\label{thm:sect2iii}
      \( \VCAT(p, \cat V) \) is a fully faithful lax epimorphism.
    \item\label{thm:sect2iv}
      \( \Cauchy p \) is an equivalence.
    \item\label{thm:sect2v}
      \( \VCAT(p, \cat V) \) is an equivalence.
 \end{enumerate}
\end{theorem}
\begin{proof}
  Combining Propositions \emph{\ref{lemma:ff}} and \emph{\ref{lemma:le}}, we
  obtain the equivalence of \emph{\ref{thm:sect2i}}, \emph{\ref{thm:sect2ii}}
  and \emph{\ref{thm:sect2iii}}. 

  Of course, \emph{\ref{thm:sect2v}} \( \to \) \emph{\ref{thm:sect2iv}}, since
  tiny objects are preserved by equivalences. Moreover, we have
  \emph{\ref{thm:sect2iv}} \( \to \) \emph{\ref{thm:sect2ii}} and
  \emph{\ref{thm:sect2v}} \( \to \) \emph{\ref{thm:sect2iii}} \textit{a
  fortiori}, taking into account that, for adjoints in $\VCat$, fully
  faithfulness coincides with $\cat V$-fully faithfulness.

  We have \emph{\ref{thm:sect2iii}} \( \to \) \emph{\ref{thm:sect2v}}, as \(
  \LKan_p \adj \VCAT(p,\cat V) \), and \(\VCAT(p,\cat V) \) is a fully
  faithful lax epimorphism. 
\end{proof}

%% file: TeX/characterization1.tex
Sobral provided a characterization of effective \( \CAT(-, \Set) \)-descent
and \( \CAT(-, \Set) \)-descent functors~\cite{MR2107401}. We show, herein,
how we can extend her characterization to the case of split fibrations. We
start by extending our characterization of fully faithful and lax epimorphic
morphisms in $\Set\dash\Cat$.

\begin{proposition}
  \label{prop:lax-epimorphism-split-fibration}
  A functor \(p \colon e \to b\) between small categories is fully faithful
  (resp. lax epimorphic) if, and only if,  \( \CAT(p,\Cat) \) is lax
  epimorphic (resp. fully faithful).
\end{proposition}
\begin{proof}
  The $2$-functor \( J \colon \Set\dash \CAT \to \Cat \dash \CAT \), taking every category to the corresponding locally discrete 2-category, is a full
  $2$-functor, and it has left and right 2-adjoints. Hence,  it	preserves and
  reflects fully faithful morphisms and lax epimorphisms by
  \cite[Lemma~2.8]{lucatelli_sousa} and  \cite[Remark~2.5]{lucatelli_sousa}.
	
  By  Propositions \ref{lemma:ff} and \ref{lemma:le}, $J(p)$ is fully faithful
  (resp. lax epimorphic) if and only if $\Cat\dash\CAT \left( J(p),
  \Cat\right) \cong \CAT \left( p, \Cat\right)  $ is lax epimorphic (resp.
  fully faithful).
\end{proof}


\begin{theorem}
  \label{thm:section3}
  Let \( p \colon e \to b \) be a functor between small categories. Denoting
  by $\CKomp{\EQP}$ the comparison functor of the codescent category of the
  factorization \eqref{eq:descent-factorization-introduction-codescent}, the
  following statements are equivalent.
  \begin{enumerate}[label=\emph{\roman*.},noitemsep]
    \item 
      $\CKomp{\EQP}$ is lax epimorphic (resp. fully faithful lax epimorphic);
    \item 
      $p$ is of \( \CAT(-,\Set) \)-descent (resp. effective \( \CAT(-,\Set)
      \)-descent); 
    \item 
      $p$ is of \( \CAT(-,\Cat) \)-descent (resp. effective \( \CAT(-,\Cat)
      \)-descent);
    \item 
      $\Cauchy {\CKomp{\EQP}}$ is lax epimorphic (resp. an equivalence)    
  \end{enumerate}   
\end{theorem}
\begin{proof} By Lemma \ref{lemma:fundamental-basic}, putting ${\cat F}=\CAT(-,\Cat):\Cat^{\op}\to \CAT$, we know that $p$ is of $\CAT(-,\Cat)$-descent (resp. effective $\CAT(-,\Cat)$-descent) if and only if $\CAT(\CKomp{\EQP},\Cat)$   is fully faithful (resp. an equivalence). By Proposition \ref{lemma:le} and Theorem \ref{thm:sect2}, this is equivalent to i., and also to iv. Using a similar argumentation for the functor ${\cat F}=\CAT(-,\Set)$ and Proposition \ref{prop:lax-epimorphism-split-fibration}, we conclude that ii. and iii. are equivalent.
\end{proof}

  
 

%% file: TeX/Beck.tex
\begin{remark}[A word on Beck-Chevalley]
  The indexed category of discrete fibrations (equivalently, the indexed
  category of split fibrations) is particularly interesting because it
  provides us with a source of counterexamples in descent theory -- it is a
  fruitful example of an indexed category (coming from a bifibration) that
  does not satisfy the Beck-Chevalley condition.

  Although it is simple to directly verify that $\CAT\left( -, \Set \right) $
  and $\CAT\left( -, \Cat \right) $ do not satisfy the Beck-Chevalley
  condition, we can do that indirectly: by showing that the
  B\'{e}nabou-Roubaud theorem does not hold for these cases (and, hence, BC
  does not hold as well).

  More precisely, we know that every effective $\CAT\left( -, \Cat
  \right)$-descent morphism induces a monadic functor (by
  \cite[Theorem~4.7]{MR4266479}). However, we can show that there are functors
  $p$ such that $\CAT\left( p, \Cat \right)$ is monadic but $p$ is not of
  effective $\CAT\left( -, \Cat \right)$-descent.

  For example, let \(p \colon \mathsf 1 \to b \) be a functor, where
  $\mathsf{1}$ is the terminal category. Note that \( p^* = \Cat(p,\Cat) \)
  is monadic whenever \(b\) has only one object, but \( p^* \) is an 
  effective $\CAT\left( -, \Cat \right)$-descent morphism if and only if $p$
  is an equivalence. To see this, note that \(p\) is of effective \(
  \CAT(-,\Cat) \)-descent iff \( \CAT(p,\Cat)\) is an
  equivalence.\footnote{This actually holds more generally. See, for
  instance, \cite[Proposition~4.3]{MR4266479}.} Therefore, by
  Proposition~\ref{prop:lax-epimorphism-split-fibration}, we conclude that \(p
  : \mathsf{1}\to b \) is of effective $\CAT(-, \Cat)$-descent if and only if
  $p$ is fully faithful lax epimorphic: and, of course, since the domain of
  $p$ is terminal, this is equivalent to $p$ being an equivalence.

  We refer to \cite[Remark~3]{MR2107401} and \cite[Remark~4.8]{MR4266479} for
  more examples of morphisms inducing monadic functors that are not of
  effective $\CAT\left( -, \Cat \right) $-descent (effective $\CAT\left( -,
  \Set \right) $-descent).
\end{remark}

%% file: TeX/final.tex
The main contribution of the present work was showing that, from the formal
observations on codescent of Section
\ref{section:Grothendieck-Descent-Codescent} and the characterization of fully
faithful lax epimorphisms of Sections
\ref{section:Fully-faithful-lax-epimorphisms} and
\ref{section:Discrete-and-Split-Fibrations}, we were able to extend Sobral's
characterization of discrete fibrations for the case of split fibrations.

The authors also believe that the present approach can be insightful towards
the study of effective descent morphisms w.r.t. some other interesting indexed
categories defined in $2$-categories.  We give two examples below.

The most natural line of work following this would be the study of effective
descent morphisms in \( \VCat \). By the observations of the present paper,
this would solely rely on the thorough study of the codescent object of
\eqref{eq:descent-factorization-introduction-codescent} in \( \VCat \) and
its Cauchy completion.

More interestingly, discrete fibrations in the context of \( (\cat T,\cat V)
\)-categories (and, more precisely, in the context of \cite{ MR1957813, MR2491799})
provides us with the indexed category $\cat E $ of \'{e}tale morphisms (see,
for instance, \cite{PhDPIER, MR3606495}).  Even for
$\cat V$ thin, the study and characterization of effective $\cat E $-descent
morphisms in this setting  is still generally an open problem.

By the present work, we can extend Sobral's techniques to this general
setting. Although we leave it to future work, we roughly sketch the general
ideas below (we refer to \cite{MR2491799} for the basic
definitions related to these brief comments).

The first step would be to characterize the lax epimorphisms in that context.
We conjecture that these are precisely the $(\cat T,\cat V)$-functors $(X,a
) \to (Y,b)$ such that $f_\ast \cdot f^\ast = b $ (following the notation of
\cite[p.~188]{MR2491799}) for $\left(\cat T, \cat V\right)$-bimodules.
A next step would be to fully study the codescent factorization
in the $2$-category of $(\cat T,\cat V)$-categories.  More precisely, this
study consists of constructing the suitable codescent objects (if/when it
exists).
Moreover, provided with the work developed in  \cite{MR2491799}, we can also
study the relation of this characterization with the Cauchy completion in this
setting.

%% file: TeX/acknowledgment.tex
We thank Maria Manuel Clementino, George Janelidze and Manuela Sobral for
all the fruitful discussions and lessons in descent theory. 

We realized this work during our research stay at MFO, Oberwolfach, in 2022.
We want to thank Stephan Klaus, Tatjana Ruf, Gerhard Huisken, Matthias Hieber,
Silke Okon, Andrea Schillinger, Verena Franke, Jennifer Hinneburg, Ronja
Firner, Charlotte Endres, Anton Herrmann,  Stefanie Reith, Petra Lein, who
warmly hosted us at the Mathematisches Forschungsinstitut Oberwolfach.